\documentclass[reqno,11pt]{amsart}
\usepackage{amssymb,latexsym}
\usepackage{amsfonts}
\usepackage{color}

\theoremstyle{plain}
\newtheorem{theorem}{Theorem}
\newtheorem{proposition}{Proposition}
\newtheorem{corollary}{Corollary}
\newtheorem{lemma}{Lemma}
\numberwithin{equation}{section}

\newcommand{\rr}{\mathbb{R}}

\newcommand{\wh}{\widehat}
\newcommand{\p}{\partial}
\newcommand{\ee}{\varepsilon}

\newcommand{\s}{\sigma}

\newcommand{\diag}{\mbox{diag}}
\newcommand {\supp}{{\textrm{supp\,}}}

\begin{document}
%\begin{titlepage}

\title[Nonuniform Dependence in Compressible Gas Dynamics]{
Nonuniform dependence on initial data for compressible gas dynamics: The Cauchy problem on $\rr^2$}
\author[J. Holmes, B. Keyfitz, F. T\i\u glay]%
{ John Holmes, Barbara Keyfitz and Feride T\i\u glay}

\date{\today}

\begin{abstract}
The Cauchy problem for the two dimensional compressible Euler equations with data in the 
Sobolev space $H^s(\mathbb R^2)$ is known to have a unique solution of the same Sobolev 
class for a short time, and the data-to-solution map is continuous.
We prove that the data-to-solution map on the plane  
is not uniformly continuous on any bounded subset of Sobolev class functions.
\end{abstract}

\maketitle
 %  
%%%%%%
%
% SECTION
%
%%%%%%   
\section{Introduction} 
In this paper, we consider the Cauchy problem for the two-dimensional 
  compressible Euler equations with data in the Sobolev space $H^s(\mathbb R^2)$.
The problem can be written in the form 
\begin{eqnarray}
\begin{cases}
   \rho_t + \rho_0 u_x +  (\rho u)_x+\rho_0 v_y  + (\rho v)_y=0 \label{eq:classical1}\\
 u_t +uu_x+ vu_y+ h_x+ \frac{ h_0 + h}{\rho_ 0 + \rho}\rho_x=0  \\
 v_t +uv_x+ vv_y+ h_y+ \frac{ h_0 + h}{\rho_ 0 + \rho}\rho_y=0  \\
   h_t+u   h_x+ v   h_y+(\gamma -1)(  h_0 + h) (u_x+v_y)   =0  
   \end{cases}\\
   \rho|_{t=0} = \phi_1   , \quad 
 u |_{t=0} = \phi_2    , \quad 
 v|_{t=0} =\phi_3    , \quad 
   h  |_{t=0} = \phi_4  ,
\end{eqnarray} 
where $\gamma>1$, $\rho_0>0$ and $h_0>0$ are constant. 
In order to arrive at this from the standard form of the equations for ideal
compressible gas dynamics (see for example Majda, \cite[pp 3-4]{m}),
we have written the density as $\rho_0 + \rho$ and have replaced the
pressure $p$ by a multiple of the internal energy, $h_0+h=p/(\rho_0+\rho)$.
The velocity components are $u$ and $v$. 
We have also written the system in nonconservative form, as we are
considering only classical solutions in this paper.
The purpose of the constants $\rho_0$ and $h_0$ is to allow us to work
with a state variable $U=(\rho,u,v,h)$ whose components lie 
in the Sobolev space, $  H^s(\rr^2) = H^s $, defined as 
$$
H^s = \left\{ f \in \mathcal S'(\rr^2)  : 
\|  \mathcal{F}^{-1} \left( (1 + |\xi|^2 )^{s/2} \wh f  \right )\| _{L^2 (\rr^2) }  
< \infty \right\}\,.
$$ 
Pointwise restrictions on the initial data
(see discussion following the statement of Theorem \ref{thmmajda}) 
allow us to stay a positive  distance from a vacuum state.

Local in time well-posedness in the sense of Hadamard for the system 
\eqref{eq:classical1} (in $d$ space dimensions)
is well known  when $s >1+ d/2$. 
The idea of the proof goes back to G\aa rding \cite{ga}, Leray, \cite{leray},
Lax \cite{l} and Kato \cite{k}; a modern version can be found in Taylor's
monograph, \cite{taylor}.  
For a more detailed exposition of the background and for alternative
proofs, see Majda \cite{m} or Serre \cite{serre}. 
In particular, if the initial data is in the Sobolev space $H^s$, for any $s>1+d/2$, 
then there exists a unique solution for some time interval which depends upon 
the $H^s$ norm of the initial data, and the solution depends continuously on the 
initial condition. 
In addition, the solution size (in $H^s$) is bounded by twice 
the size of the initial condition for some period of time. 
Classical solutions to the compressible Euler equations do not
exist globally in time. 
Indeed, it has been shown that even for 
almost constant initial data, there is generally a critical time, $T^C$, at which the 
classical ($H^s$) solution breaks down \cite{m}. 
This breakdown is characterized by the formation of shock waves; that is, as $t \nearrow T^C$,
$$
\limsup_{t\rightarrow T^C} \| u_t \|_{L^\infty} + \| \nabla u\|_{L^\infty} = \infty. 
$$

Weak solutions for quasilinear systems in conservation form (the standard
form for \eqref{eq:classical1}) have been extensively studied in a single
space dimension, where there is a complete well-posedness theory for
data of small total variation (and in some cases small oscillation).
Excellent monographs by the originators of this theory can be found in
Bressan \cite{bressan} and Dafermos \cite{dafermos}.

An outstanding open problem in multidimensional hyperbolic conservation laws 
is to develop a theory of weak solutions for times after the formation of a shock wave. 
This is an active area of current research. 
\v Cani\' c, Keyfitz, Jegdi\' c and co-authors (for example, \cite{CKK:WRR, CShk, JKC}
and the recent \cite{JKC00}) have looked at self-similar solutions of two-dimensional
problems, as have Chen, Feldman and co-authors, \cite{ChFe,ChFe4} for example).
There is also interesting work by Shu-Xing Chen, \cite{Ch1} and other papers,
and by Elling, see \cite{Ell1} and references there.
An intriguing line of research concerns ill-posedness of multi-dimensional
problems of the type of \eqref{eq:classical1} in spaces other than $H^s$;
Rauch \cite{Rau1}, following Brenner \cite{Bren}, identified key points
of this issue, first identified by Littman \cite{Lit}; and
Dafermos \cite{Da1} and Lopes \cite{LL} have followed it up.
Yet another question, that concerns the proper definition of weak
solutions, is raised by the ``wild'' weak solutions of
De Lellis and Sz\' ekelyhidi \cite{DLS}. 
While this does not seem to bear on the question we tackle here,
which concerns classical solutions, it is worth mentioning, both as
a note about well-posedness, and as evidence of the relationship
between the compressible and incompressible gas dynamics 
equations, which we exploit in this paper.
 
The compressible Euler equations can be reduced, in the zero Mach number limit, 
to the incompressible Euler equations (see \cite{m} or \cite{mb}
for details on the asymptotic analysis), 
\begin{equation} \label{incomp}
\begin{split}
 u_t +uu_x+ vu_y  +p_x&=0  \\  
 v_t +uv_x+ vv_y +p_y&=0  \\
  u_x + v_y &= 0 ,
\end{split}\end{equation}
where $p$ is pressure.  
Global in time well-posedness is also an important question for the 
incompressible Euler equations. 
For a summary of the open questions, we refer the reader to Constantin \cite{c} 
and Fefferman \cite{f}. 
For local well-posedness and related results see Majda--Bertozzi \cite{mb}.
Because the incompressible system is not hyperbolic, analysis of the two 
problems -- \eqref{eq:classical1} and \eqref{incomp} -- has proceeded
along rather different lines.
This paper finds a rather striking connection.

A point of departure for our analysis is the proof of the non-uniform continuity of 
the data-to-solution map for the incompressible Euler equations
recently established  by Himonas and Misio\l ek \cite{hm}. 
In particular, in dimensions two and three they found solutions for 
periodic data and for Sobolev space data, for which the data-to-solution map 
was not uniformly continuous. 
In the non-periodic (full plane) case, their method used a technique of high-low 
frequency approximate solutions developed by Koch and Tzvetkov \cite{kt} for 
the one-dimensional Benjamin-Ono equation. 
Our main result is to show that, in a similar way, 
dependence on the initial data is not better than 
continuous for classical solutions of the compressible Euler equations. 
We state our result as follows. (Here we assume the standard restriction on
$s$, $s>d/2+1$.)
 
\begin{theorem}\label{intro-thm} 
For $s>2$,
the data-to-solution map for the system \eqref{eq:classical1} is not uniformly 
continuous from any bounded subset of 
$ ( H^s(\mathbb R^2)  )^4  $ to the solution space
$ C([-T, T]; ( H^s(\mathbb R^2)  )^4   ) $.  
\end{theorem} 

Our proof of non-uniform dependence of the data-to-solution map uses a  
method similar to that of \cite{hm} and \cite{kt}: construction
of high-low frequency approximate solutions. 
We formulate a different way of defining the low frequency terms. 
 In particular, Koch-Tzvetkov and Himonas-Misio\l ek 
 use an $L^2$ energy estimate, while we use an energy estimate
 in $H^\s$, $ \s < s-1$. 
 We are able to do this by sidestepping 
 the construction of some low frequency exact solutions to the compressible 
 Euler equations. 
 The strategy in this paper is to find
estimates in the $H^\s$ norm for $\s $ near $s$.
We find that the low frequency 
residual terms actually {\em help} to give the desired estimates by allowing for a crucial cancellation. 
These convenient  cancellations, obtained in our construction, simplify  technical 
difficulties created by the more complicated system of equations. 
The construction of approximate solutions and demonstration of non-uniformity
were first carried out in the
ideal compressible gas dynamics system \eqref{eq:classical1} for
periodic data.
This result is in the companion paper of Keyfitz and T\i\u glay \cite{KeTi} along with 
the description of the flow for the approximate solutions that we use.

Continuity properties of the data-to-solution map for a  variety of equations have 
been studied by   many other authors. In particular, 
the first result of this type was shown by Kato \cite{k} for Burgers' equation, 
$
u_t + (u^2)_x = 0$. 
Kato showed that the data-to-solution map is not H\"older continuous from any 
bounded subset of $H^s$ to $H^s$, when $s>3/2$.  

The  idea of using high-frequency approximate solutions has also been employed 
extensively in the context of dispersive equations.   
For example, both Christ, Colliander and Tao \cite{cct} and Kenig, Ponce and Vega 
\cite{kpv} used a similar   method of  high-frequency approximate solutions to show 
ill-posedness of some defocusing dispersive equations.   
This methodology was also adapted by  Himonas and Kenig \cite{hk} for the 
Camassa-Holm (CH) equation on the circle, and by Himonas, Kenig and Misio\l ek 
\cite{hkm} for the CH equation in the non-periodic case. 
For additional related results concerning the continuity of data-to-solution maps, 
we refer the reader to  Bona and Tzvetkov \cite{bt},  Holmes \cite{h},  
Molinet, Saut and Tzvetkov \cite{mst} and the references contained therein. 

In the next section, we give some preliminary results and notation which we 
shall use throughout our proof. 
Section \ref{three}  gives the proof of non-uniform dependence. 

%%%
%
% Preliminaries 
%
%%%
\section{Preliminary Results and Notation}
This section summarizes background needed in the 
rest of the paper. 
The operator $\Lambda^s f $ is defined by the formula
\[
\wh{ \Lambda^s f } (\xi, \eta) = (1+ \xi^2 + \eta^2 )^{s/2}  \wh f (\xi, \eta)\, ,
\]
where $f$ is a test function.
Here $s$ may be any positive real number; in order to use the standard
existence theorems for classical solutions of \eqref{eq:classical1}, we 
take $s>d/2+1=2$.
The notation $\widehat{\mbox{\phantom{xx}}}$ stands for the usual Fourier transform.
 The Sobolev space $H^s    $ is a Hilbert space equipped with inner product
 and norm given by
\[
\| f\|_{s   } ^2  =   \langle  \Lambda^s f , \Lambda^s f \rangle _{L^2  }.
\]
We will frequently employ the following Sobolev embedding theorem 
(see, for instance, Taylor \cite[p 272]{taylor}).
\begin{theorem}[Sobolev embedding] \label{sob}
If $s > k + 1 $, then $H^s$ is continuously embedded into $C^k$: 
$H^s  \hookrightarrow C^k  $. 
Specifically,
\begin{align}
 H^s \subset \{ f \in C^k : D^\alpha f (x_1,x_2) \rightarrow 0 \text{ as } 
 | (x_1,x_2)| \rightarrow \infty, |\alpha | \le k \} ,
\end{align} 
and the inclusion is continuous;  for some constant $C(s,k)$ we have 
\begin{align}
\| f\|_{C^k } \dot =  \sum_{ |\alpha | \le k } \| D^\alpha  f \| _{L^\infty} \le C(s,k) \| f\|_{s}. 
\end{align} 
\end{theorem}
We will also liberally employ the following    
classical product estimate (see for instance Taylor   \cite[p 66]{taylor}).
\begin{lemma}  \label{sli-algebra}
If $s > 0$ and $f,g \in L^{\infty} \cap H^s $,   we have the estimate 
\begin{align*}
\| f g \|_{s  } \le  C(s)\left[  \| f\| _{L^\infty  }  \|g\| _{s} + \| f\| _{s  }  \|g\| _{L^\infty} \right] \, .
\end{align*}
\end{lemma} 
This, combined with the Sobolev embedding theorem,  implies that $H^s$ is a 
Banach algebra whenever $s > 1$; in other words, for $f$, $g \in H^s$, 
the product $fg \in H^s$. 
Moreover, we have the algebra estimate
\begin{align}
\| f g \|_{s  } \le C(s) \| f\| _{s  }  \|g\| _{s}    . \label{srp}
\end{align}  

 For any test function $f$,
the commutator operator $[ \Lambda ^s , f]$ applied to a test function $g$    is 
\begin{align} \label{commdef}
[ \Lambda ^s , f] g  = \Lambda ^s (f g) - f \Lambda^s g\,. 
\end{align}  
The following commutator estimate can be found in Kato and Ponce \cite{kp}. 
\begin{theorem}[Kato-Ponce commutator estimate]
If $s \ge 0 $, and $f\in Lip\cap H^s$ and $g\in L^{\infty} \cap H^{s-1}$, then 
\begin{align} \label{KPest}
\| [ \Lambda ^s , f] g \| _{L^2}  \le C(s)\left(  \|  \p f \| _{L^\infty} \| \Lambda^{s-1} g
 \|_{L^2} + \| \Lambda^s f \|_{L^2} \|g \| _{L^\infty} \right) .
\end{align} 
\end{theorem}
 In the proof of Lemma \ref{error} we need a simple interpolation estimate:
 \begin{proposition} \label{interp}
 If $u\in H^\tau$ and $\s<s<\tau$ then
 $$ 
 \big\|u\big\|_{ {s} }  \leq\big\| u\big\|_ \s^{\alpha}  
\big\| u\big\|_\tau^{\beta}, \quad \textrm{where}
\quad \alpha= \frac{ \tau -s}{\tau - \s}\,,\quad \beta =\frac{  s -\s }{\tau - \s} \,.
$$
 \end{proposition}
 \begin{proof}
 We write
 $$ \big\|u\big\|_{ {s} }^2  =\int \big(1+|\xi|^2\big)^s\left|\hat{u}(\xi)\right|^2\,d\xi
 =\int\left[ \big(1+|\xi|^2\big)^{\s}\left|\hat{u}(\xi)\right|^{2}\right]^\alpha
\left[ \big(1+|\xi|^2\big)^\tau\left|\hat{u}(\xi)\right|^2\right]^\beta\,d\xi
 $$
 and apply H\"older's inequality with $p=1/\alpha$ and $q=1/\beta$.
 \end{proof}
Finally, owing to the nature of the nonlinearities in \eqref{eq:classical1}, 
we need the following reciprocal estimate.
It was proved by Kato, \cite[Lemma 2.13]{k}, for functions in ``uniformly local'' Sobolev
spaces (which generalize our construction of coefficients of the form $\rho_0 + \rho$),
and for integer values of $s>2$. We provide a sketch of the proof in the delicate case when $1<s<2$; the larger value of $s$ are straightforward. 
\begin{lemma} \label{Reciprocal estimate}
If $s>1    $,  $ h \in H^{s } (\rr^2) $, $g  \in H^{s } (\rr^2)  \cap C^1(\rr^2)$ and 
$b >0$ is a constant such that $g+b > \frac12 b$, then 
\[
\left\| \frac{h}{g+b}  \right\| _{s} \le  C(s,b ) \left(1+\|   g\| _{C^1}^s  +\|g\|^{s}_{{s } }\right)      
   \|h \|^{}_{{s } } .
\]    
\end{lemma}    
\begin{proof} 
In the case $s = 1 + \gamma$ with $\gamma \in (0,1)$, the integer parts of the norm satisfy 
this bound as in Kato, \cite[Lemma 2.13]{k}. The fractional portion of the norm  (see 
\cite[page 155]{stein} for instance, for this form of the Sobolev norm) is 
$$  \sup_{|\alpha |=1}  \|  D^\alpha \left(h  (g+b)^{-1}  \right)   \|_{\dot  H^{\gamma} }  ^2 
 \equiv  \sup_{|\alpha |=1} \int_{\rr^4}  \frac{ | D^\alpha  \left( h(g+b)^{-1}  \right)({\bf x}) -
 D^\alpha \left( h(g+b)^{-1}  \right)({\bf y}) |^2 }{|{\bf x}-{\bf y}|^{2\gamma +2 } } 
 d{\bf x}\,d{\bf y}\,,$$
where ${\bf x}$ and ${\bf y}$ are points in $\mathbb R^2$ and $\dot  H^{\gamma} $ is 
the homogeneous Sobolev space. 
Consider $D ^\alpha= \p_1$ (the partial derivative with respect to the first component) so that 
\begin{equation}\label{parta27} 
\p_1  \left( h(g+b)^{-1} \right) =   (g+b)^{-1} \p_1 h - h (g+b)^{-2}  \p_1 g .
\end{equation}
 Estimating the first term on the right hand side of \eqref{parta27} is a straightforward 
 calculation after  breaking the integral into the following two pieces 
\begin{align}
 \|  (g+b)^{-1}  \p_1 h \|_{\dot  H^{\gamma} } ^2 \le&\,  2
 \int  \frac{ 1} { |  g ({\bf x}) +b  | ^2}  \frac{  
 |      \p_1 h  ({\bf x})  - \p_1 h({\bf y})      | ^2
 }{|{\bf x}-{\bf y}|^{2\gamma +2 } } d{\bf x}d{\bf y}
 \notag
\\ 
&+ 2 \int          \frac{|       {g({\bf x}) - g({\bf y}) }    | ^2
}{|{\bf x}-{\bf y}|^{2\gamma +2 } }   \frac{ 1   }{|g({\bf x})+b  |^2  }\,  d{\bf x}   
 \frac{ |\p_1  h({\bf y}) |^2}{|g({\bf y})+b |^2}\,d{\bf y} 
\label{int_3}\,.
\end{align} 
The first integral is bounded by an application of H\"older's inequality, while the second  
term  additionally requires the Sobolev embedding theorem and the following calculus 
estimate: 
$$
\sup_{y\in \rr^2} \int  \frac{  
 |       g({\bf x}) - g({\bf y})     |  ^2 
}{|{\bf x}-{\bf y}|^{2\gamma +2 } }\, d {\bf x}     \le C(\gamma)\|g\|_{ C^1  }^2  .
$$ 
This estimate is obtained by splitting the domain of integration into two pieces, 
$|x-y|<1$ and $|x-y|\ge1$, and then applying the mean value theorem. 
Returning to equation \eqref{parta27},
 the second term on the right hand side is  bounded by Lemma \ref{sli-algebra} and the 
 Sobolev embedding theorem:
\begin{align*}
\|h (g+b)^{-2}  \p_1 g\|_{    {\gamma} }  &\le 
\|  (g+b)^{-2}  \p_1 g\|_{   {\gamma} } \| h \|_{L^\infty}  +\|  (g+b)^{-2}  \p_1 g\|_{  L^\infty  } \| 
h \|_{  {\gamma}}  
 \\& \le 
\left( \frac{4}{b^2 }  \|    g\|_{   s   }  + \|  (g+b)^{-2}  \p_1 g\|_{\dot  H^{\gamma} }  \right) \| 
h \|_{ s }  + \frac{4}{b^2} \|     g\|_{  C^1    } \| h \|_{  {s}} .
\end{align*}  
We   bound $ \|  (g+b)^{-2}  \p_1 g\|_{\dot  H^{\gamma} } $ 
in the same way as $ \|  (g+b)^{-1}  \p_1 h \|_{\dot  H^{\gamma} } $. 
The same estimates hold for $\partial_2$. \end{proof}
%%%
%
% The non-periodic case
%
%%%
\section{Proof of Nonuniform Dependence}
\label{three}
We write the compressible Euler system \eqref{eq:classical1} in the form 
\begin{align}
 U_t+A(U)U_x+B(U)U_y=0, \label{first-system}
\end{align}
with  $U= ( \rho,u,v,h)^T$
and
\[
 A(U)=\begin{pmatrix}
u &\rho_ 0 + \rho & 0 & 0 \\ 
\frac{ h_0 + h}{\rho_ 0 + \rho} & u & 0 & 1 \\
0 & 0 & u & 0 \\
0 &( \gamma-1)( h_0 + h )& 0 & u
\end{pmatrix}\,,\quad
B(U)=  \begin{pmatrix}
v & 0 &\rho_ 0 + \rho & 0 \\ 
0 & v & 0 & 0 \\
\frac{ h_0 + h}{\rho_ 0 + \rho}& 0 & v & 1 \\
0 & 0 & (\gamma-1)(h_0 + h) & v
\end{pmatrix}\, .
\] 

\subsection{Symmetrized system} \label{symmsyst}
The system \eqref{first-system} is symmetrizable; that is, it can be written as 
\[ A_0 U_t+A_1(U)U_x+B_1(U)U_y= 0\,,
\]
where the matrices $A_0, A_1, B_1$ are symmetric and 
$A_0$ is positive definite.
We can choose  
\begin{align*}  
&A_0(U) =  \begin{pmatrix}
\frac{ h_0 + h}{\rho_ 0 + \rho} & 0 & 0 & 0 \\ 
0 & \rho_ 0 + \rho & 0 & 0 \\
0 & 0 & \rho_ 0 + \rho & 0 \\
0 &0 & 0 & \frac{\rho_ 0 + \rho}{(\gamma -1 ) (h_0 + h)}
\end{pmatrix}\,;
\\
 &   
 A_1(U)=  \begin{pmatrix}
\frac{ h_0 + h}{\rho_ 0 + \rho} u & h_0 + h & 0 & 0 \\ 
h_0 + h &  ( \rho_ 0 + \rho )  u& 0 &  \rho_ 0 + \rho  \\
0 & 0 &  (\rho_ 0 + \rho)u & 0 \\
0 & \rho_ 0 + \rho & 0 & \frac{(\rho_ 0 + \rho) u}{(\gamma-1) (h_ 0 + h)}
\end{pmatrix} \,;
\\
&  
B_1(U)= \begin{pmatrix}
\frac{ h_0 + h}{\rho_ 0 + \rho} v& 0 & h_0 + h & 0 \\ 
0 & ( \rho_ 0 + \rho ) v & 0 & 0 \\
 h_0 + h & 0 & (\rho_ 0 + \rho) v & \rho_ 0 + \rho \\
0 & 0 & \rho_ 0 + \rho & \frac{(\rho_ 0 + \rho) v}{(\gamma-1) (h_ 0 + h)}
\end{pmatrix} \, .
\end{align*}

\subsection{Approximate solutions}
 
 Our strategy, following the template laid out by Himonas and Misio\l ek \cite{hm},
 is to use two sequences ($\omega=\pm 1$) of approximate solutions:
\begin{equation} \label{Udef}
U^{\omega,n}  = \begin{pmatrix}
 \rho^{\omega,n} \\ 
u^{\omega,n}\\
v^{\omega,n}\\
  h^{\omega,n}
\end{pmatrix} 
 = 
\begin{pmatrix}
 0  \\ 
u_1+ u_2 \\
v_1+v_2 \\
0
\end{pmatrix} \,.
\end{equation} 
The approximate solutions contain low frequency functions $u_1, v_1$ 
and  high frequency functions $u_2$ and $v_2$. 
(Our notation suppresses, for clarity, the dependence of the
$u_i$ and $v_i$ on $n$ and $\omega$.)
The high frequency functions are defined for a constant $\delta>0$ as 
\begin{equation} \label{defu2}
 u_2 = n^{-\delta - s - 1 } \p_y S ,
 \text{ and }
 v_2 =  -    
n^{-\delta - s - 1 } \p_ x S
 ,
\end{equation}
where
$S$ is a   stream function, given by
$$
S(x,y,t) = \psi ( n^{-\delta } x) \psi(n^{-\delta} y) \sin( n y  + \omega t),
$$
for a  compactly supported nonnegative cutoff function
  $\psi  $ which equals one on $[-2,2]$.  
Expanding $u_2$ and $v_2$ gives 
\begin{equation} \label{highfreq}
\begin{split} 
& u_2 =n^{-2\delta -s-1}\psi(n^{-\delta} x)\psi'(n^{-\delta} y) \sin( n y + 
\omega t) +  n^{-\delta -s}\psi(n^{-\delta} x)\psi(n^{-\delta} y) \cos( n y + \omega t)
\\
& v_2 =-n^{-2\delta -s-1}\psi'(n^{-\delta} x)\psi(n^{-\delta} y) \sin( n y+ \omega t) .
\end{split}
\end{equation} 
The low frequency functions, $u_1$ and $v_1$, are
\begin{equation} \label{lowfreq}
u_1 = \omega n^{-1 } \varphi _1\left( n^{-\delta} x \right )
 \varphi '_2\left( n^{-\delta} y \right )\,,  \text{ and } 
v_1 =-\omega n^{-1 } \varphi _1' \left( n^{-\delta} x \right ) 
\varphi _2\left( n^{-\delta} y \right ) ,
\end{equation}
where $\varphi _1' $ and $\varphi _2  $   are also smooth compactly supported functions;
$\varphi_1' $ is  identically 
$1$ on the support of $\psi'   $ and $\varphi_2\equiv 1$ on $\supp\psi $. 
The following cancellation holds.
\begin{lemma}\label{lemma-div-uv}  
For $u$ and $v$ defined in \eqref{Udef}, \eqref{highfreq}, and 
\eqref{lowfreq}, we have
$\p_x u^{\omega, n} + \p_y v^{\omega,n} = 0.$ 
\end{lemma} 
\begin{proof} 
We have
$$ u_{1,x}=\frac{\omega}{n^{1+\delta}}\phi_1'\left(\frac{x}{n^\delta} \right)
\phi_2'\left(\frac{y}{n^\delta} \right) = -v_{1,y}\,.$$
Considering the high frequency terms, we see from \eqref{defu2} 
\begin{equation*}
\p_x u _2  + \p_y v_2 = n^{-\delta - s - 1 } \p_x  \p_y S  
-n^{-\delta - s - 1 }\p_y  \p_ xS  = 0 .       
\qedhere
\end{equation*}
\end{proof}
As a result of the definition, the approximate solutions satisfy
\[ U_t^{\omega,n} +A(U^{\omega,n} )U_x^{\omega,n} +B(U^{\omega,n} )U_y^{\omega,n} =R,
\]
where 
\begin{equation} \label{Rdef}
R =
\begin{pmatrix}
0  \\ 
R_2 \\
R_3  \\
0
\end{pmatrix} =
\begin{pmatrix}
0\\
 \p_t u^{\omega, n} +  u^{\omega, n} \p_x u^{\omega, n}   
 + v^{\omega, n}  \p_y u^{\omega, n} \\
 \p_t v^{\omega, n} + u^{\omega, n}\p_x v^{\omega, n} 
 + v^{\omega, n}  \p_y v^{\omega, n}     \\0
\end{pmatrix}  \,. \end{equation}
 
 Denote the inner product of two vectors, $V$ and $ W$, by  
 $ \langle V, W\rangle = \sum \langle V_i, W_i\rangle _{L^2} $, and
for any vector $U$  denote 
\begin{align}
 \| U\|_{  \s}^2 =\langle \Lambda^{\sigma}  U , \Lambda^{\sigma}  U \rangle = 
 \| \rho \|_{\s} ^2 +  \| u \|_{\s} ^2+ \| v \|_{\s} ^2+ \| h \|_{\s} ^2.
\end{align}
 Let $U_{\omega, n}  = ( \rho _{\omega, n}  ,
  u _{\omega, n}   ,
  v _{\omega, n}   ,
  h _{\omega, n}   )^T$ be the actual solution to  the Cauchy problem 
  corresponding to equation \eqref{first-system},
  with the same data:
  $$U_{\omega,n}(x,y,0)=U^{\omega,n}(x,y,0)=
  \big(0, u_1(x,y,0)+u_2(x,y,0),v_1(x,y,0)+v_2(x,y,0),0\big)\,,$$
  again with $\omega = 1$ or $-1$. 

The actual solution is unique, and exists on a time interval 
which depends only upon the size (in the $H^s$ norm)
of the initial data and on its distance from the boundary of the
region of state space
(called $G$ in the statement below) where the system is hyperbolic.
We quote the following theorem found in \cite{m}. 
\begin{theorem}[Majda, Theorem 2.1] \label{thmmajda}
Assume $U(\cdot,0)=U_0 \in H^s$, $s>d/2+1$ and $U_0(x) \in G_1$, 
$\bar G_1  \subset\subset G$. 
Then there is a time interval $[0,T]$ with $T>0$, so that the equations 
\eqref{eq:classical1} have a unique classical solution $U  \in C([0,T]; 
(H^s)^4 ) \cap C^1 ( [ 0,T];(H^{s-1})^4  ) $, and $U(x,t)\in G_2$, $\overline G_2
\subset\subset G$,  for $(x,t) \in \rr^2 \times [0, T]$;
here $T=T(\|U_0 \|_{ s}, G_1)$. 
\end{theorem}
In our coordinate system, $G=\{\rho>-\rho_0\}$.
Having specified values for $\rho_0>0$ and $h_0$, we might choose,
for example, data to lie in a bounded set
$$ G_1=\left\{-\frac{1}{4}\rho_0 < \rho(\cdot,0)< M_\rho, |u(\cdot,0)|<
M_u,  |v(\cdot,0)|< M_u, -\frac{1}{4}h_0 < h(\cdot,0)< M_h\right\}\,,
$$
and then take $G_2$ to be
$$ G_2=\left\{-\frac{1}{2}\rho_0 < \rho(\cdot,0)< 2M_\rho, |u(\cdot,0)|<
2M_u,  |v(\cdot,0)|< M_u, -\frac{1}{2}h_0 < h(\cdot,0)< 2M_h\right\}\,,
$$
where $M_\rho$, $M_u$ and $M_h$ are positive numbers.
The significant bound, which we need throughout, is the lower bound
on $\rho$ in $G_2$.
Additionally, continuous dependence on the data
yields the following $H^s$ solution size estimate. 
\begin{theorem}[\cite{m}, Theorem 2.2] \label{solsize}
Assume $U|_{t=0} \in H^s$, $s>2$ and  $U|_{t=0} \in G_1$.  
There exists a $T^*$,  $0< T^*\le T$ such that 
$$
\sup_{t \in [0, T^*]} \| U \|_s \le 2 \| U|_{t=0}\|_{s}.
$$
\end{theorem}
In what follows, we take $T^*$ to be the value given by this theorem.

We obtain the proof of Theorem \ref{intro-thm} by
 showing the following properties of the corresponding solutions:
\begin{enumerate}
\item \label{bdd}
Boundedness of initial data (proved in Section \ref{bddcvge}): 
\begin{align}
\|U^{\omega,n}(\cdot,0)\|_s=\|U_{\omega,n}(\cdot,0)\|_s= 
\| ( \rho _{\omega, n} (0) , u _{\omega, n}(0),
 v _{\omega, n} (0), h _{\omega, n} (0)  )  \|_{s }  \le C ,
\end{align}
uniformly in $n$.
\item \label{cvge}
Convergence of initial data (Section \ref{bddcvge}): 
for $\delta<1/2$,
\begin{align}
&\lim_{n\rightarrow \infty} \| U_{1,n}(\cdot,0)- U_{-1,n}(\cdot,0 ) \|_{s } = 0.
\end{align}

\item \label{app}
Uniformity of approximation of $U^{\omega,n}$ to actual solution
$U_{\omega,n}$ (Section \ref{uniform}):
\begin{align} \label{eqapp}
\| U_{\omega,n}(\cdot,t) -U^{\omega,n}(\cdot,t) \|_{s }  
\le C n^{-\ee} , \quad    \quad 0<t< T^*\,,
\end{align}
for some $\ee>0$.
\item  \label{nu}
Non-uniformity of divergence of $U_{1,n}$ and $U_{-1,n}$ from each
other in time (Section \ref{noncvge}):
\begin{align}
&\|U_{1,n}(\cdot,t)-U_{-1,n}(\cdot,t  ) \|_{ s} > |\sin(t)|, \quad    \quad 0<t.
\end{align}
  \end{enumerate}
The following estimates  can be found in the appendix of \cite{hm}. 
\begin{lemma} \label{hs-estimates} 
Let $\s \ge 0 $, $\delta  >0$ and $n \gg 1$. 
For any Schwarz function $\psi \in \mathcal{S}(\rr)$ we have 
\begin{align} 
\| \psi ( n^{-\delta} \cdot) \|_{H^\s(\rr) } \le  n^{\delta/2} \| \psi    \|_{H^\s (\rr) } .
\end{align}
For any constant $a \in \rr$ we have 
\begin{align} \label{cos_estimate} 
\| \psi ( n^{-\delta} \cdot)  \sin(n \cdot + a) \|_{H^\s(\rr) } + 
\| \psi ( n^{-\delta} \cdot)  \cos(n \cdot + a) \|_{H^\s(\rr) } \approx n^{\s + \delta/2} \| \psi    \|_{L^2 (\rr)} .
\end{align}
\end{lemma}
The notation $\approx$ means that the expression on the left is
bounded above and below by constants
independent of $\s$, $\delta$ and $n$.
Note that the $L^2$ bound implies an $H^\s$ bound.
From this lemma we obtain bounds on the approximate solutions:
\begin{lemma} \label{u0bds}
For $s-2<\s< s-1$ and $0<\delta < 1$, we have
$$ \|U^{\omega,n}\|_{\s+1}\leq C n^{\s - s+1}\,,$$
where $C$ depends on the norms of the functions $\psi$, $\phi_1$ and $\phi_2$.
\end{lemma}
\begin{proof}
The nonzero terms in $U^{\omega,n}$ are $u_1$, $u_2$, $v_1$ and $v_2$.
Since $u_1$ and $v_1$ are products (in $x$ and $y$) of terms of the form $\psi(n^{-\delta}
\cdot)$, we have, from Lemma \ref{hs-estimates}, and with $C$ a generic 
constant,
$$ \|u_1\|_{\s+1}=n^{-1}\|\phi_1(n^{-\delta} x)\|_{\s+1} \|\phi_2'(n^{-\delta} y)\| _{\s+1}
\leq n^{-1}n^{\delta/2}\|\phi_1\|_{\s+1}n^{\delta/2}\|\phi'_2\|_{\s+1}=Cn^{-1+\delta}\,.
$$
A similar bound holds for $v_1$, which has the same structure.
Note that these bounds are valid for any $\s$.
On the other hand,
\begin{align*} 
\|u_2\|_{\s+1} \leq &n^{-2\delta -s-1}\|\psi(n^{-\delta} x)\|_{\s+1}
\|\psi'(n^{-\delta} y) \sin( n y + \omega t)\|_{\s+1}\\ 
&\phantom{word} +  n^{-\delta -s}
\|\psi(n^{-\delta} x)\|_{\s+1}\|\psi(n^{-\delta} y) \cos( n y + \omega t)\|_{\s+1}\\
\leq &n^{-2\delta -s-1}n^{\delta/2}\|\psi\|_{\s+1}n^{\s+1+\delta/2}\|\psi'\|_{\s+1}
+ n^{-\delta -s}
n^{\delta/2}\|\psi\|_{\s+1}n^{\s+1+\delta/2}\|\psi\|_{\s+1}\\
\leq& Cn^{-\delta +\s -s}+Cn^{\s-s+1}\,,
\end{align*}
while $v_2$, which has the structure of the first term in $u_2$, satisfies
$$ \|v_2\|_{\s+1}\leq Cn^{-\delta +\s -s}\,.$$
Now, $\delta$ is a positive number and $\s<s-1$, so all the exponents of
$n$ are negative if $\delta < 1$.
To bound the low frequency terms by the high frequency terms we need
$-1+\delta < \s - s +1$, or $\delta < \s-s+2$, and provided $\s>s-2$, 
as we have assumed, it is possible to achieve this with $\delta>0$.
\end{proof}
Lemma \ref{u0bds} implies a bound on the actual solution, using Theorem
\ref{solsize}.
\begin{corollary} \label{moresolsize}
If $t\leq T^*$, then
$$ \|U_{\omega,n}\|_{\s+1}\leq C n^{\s + 1-s} $$
where $T^*$, as in Theorem \ref{solsize}, 
is the time to doubling of the initial norm, and $\s< s-1$.
\end{corollary}
%%%%%%%%%proofs of (1) and (2) 
\subsection{ (\ref{bdd}) Boundedness and  (\ref{cvge})
Convergence of the Initial Data} \label{bddcvge}
 From Lemma \ref{hs-estimates}, we have $\|U^{\omega,n}(\cdot,0)\|_s\leq 
 Cn^{-1+ \delta} + C$.
For any $\delta$ with $0<\delta<1$ 
 we have $\|U^{\omega,n}(0)\|$ bounded, uniformly in $n$. 
 
 To see that the difference in the initial data for $\omega=\pm 1$ converges 
 to zero in $H^s$,
 we calculate $U^{1,n}-U^{-1,n}$ at $t=0$, noting that the oscillatory terms
 cancel at $t=0$, leaving only
$$
\begin{pmatrix}
 \rho_{1,n}  (0) - \rho_{-1,n} (0) \\
 u_{1,n} (0) - u_{-1,n}  (0)\\
v_{1,n} (0) - v_{-1,n}  (0)\\
h_{1,n} (0) - h_{-1,n}  (0)
\end{pmatrix} 
= 
\begin{pmatrix}
0\\
 2 n^{-1    } \varphi _1\left( n^{-\delta} x \right ) \varphi _2 ' \left( n^{-\delta} y \right )  \\
2n^{-1  } \varphi _1'\left( n^{-\delta} x \right ) \varphi _2\left( n^{-\delta} y \right ) \\
 0
\end{pmatrix}.
$$
This tends to zero in $H^s$ by the 
first estimate in Lemma \ref{hs-estimates}, for any $\delta\in(0,1)$,
as in  the first estimate in the proof of Lemma \ref{u0bds}. 

% proof of (3)
\subsection{(\ref{app}) Uniformity of the Approximation }  \label{uniform}
 \label{section3.4} 
 In this subsection,   we denote the actual solutions by $U$. 
  Let $\mathcal{E}  = U  - U^{\omega ,n}  = ( E, F, G, H)^T$
   be the error,
   the difference between the actual and approximate solutions. 
The main result of this section is  
\begin{theorem} \label{bigdeal}
For  $\max\{1,s-2\} < \s < s-1$, $\mathcal E$ satisfies  
$$
\frac{d}{dt} \| \mathcal{E}   \|_{\s}   \lesssim n^{\s+1-s}  \| \mathcal{E}  \|_\s  + n^{ \delta-2 } .
$$
Furthermore, we have on the time interval of existence
$$
\| \mathcal{E}  \|_{\s} \lesssim n^{  \delta-3 + s - \s }.
$$
\end{theorem}   
 
 \begin{proof}  
  An equation for the error  
 (the symmetric form of the system is useful here) is
\begin{align} \label{error-est} 
A_0 (U^{\omega,n} )  \mathcal{E} _t 
+ A_1(U^{\omega,n}) \mathcal{E} _x
+B_1 (U^{\omega,n})\mathcal{E} _y
+C(U ) \mathcal{E} 
+ A_0(U^{\omega,n}) R = 0
\end{align}
where 
\begin{align*}
C(U )   &
= A_0 (U ^{\omega, n} )  
 \left( \begin{array}{cccc}
u_x+v_y&  \rho_x & \rho _y  & 0 \\ 
 -  \frac{     h_0 \rho_x}{\rho \rho_0  + \rho_0^2 }  
 &u_x & u_y & \frac{\rho_x}{\rho+\rho_0} \\
 -  \frac{    h_0 \rho_y}{\rho \rho_0  + \rho_0^2 }  
 & v_x & v_y &   \frac{\rho_y}{\rho+\rho_0} \\
0 &h_x &  h_y &(\gamma-1) ( u_x + v_y) 
\end{array} \right)    
\\ & 
= 
 \left( \begin{array}{cccc}
\frac{ (u_x+v_y ) h_0  }{\rho_0   }&\frac{  h_0   }{    \rho_0    }  \rho_x
 &\frac{  h_0   }{    \rho_0    }  \rho _y  & 0 \\ 
 -  \frac{      h_0\rho_x}{\rho   + \rho_0  }  &  \rho_0u_x 
&   \rho_0u_y & \frac{     \rho_0\rho_x}{\rho+\rho_0} \\
 -  \frac{      h_0\rho_y}{\rho   + \rho_0  }  &  \rho_0 v_x 
&  \rho_0 v_y &   \frac{     \rho_0\rho_y}{\rho+\rho_0} \\
0 &\frac{      \rho_0   }{ (\gamma-1) h_0  }  h_x 
&\frac{      \rho_0   }{ (\gamma-1) h_0  }   h_y &\frac{     \rho_0 ( u_x + v_y)   }{  h_0   }   
\end{array} \right)   
\,.
\end{align*}
We write $C(U)\mathcal{E}$ as
\begin{align*}
C(U )   \mathcal{E}&
= 
 \left( \begin{array}{c }
\frac{ (u_x+v_y ) h_0  }{\rho_0   } E+ \frac{  h_0   }{    \rho_0    }  \rho_x
F+ \frac{  h_0   }{    \rho_0    }  \rho _y  G   \\ 
 -  \frac{    h_0\rho_x}{\rho    + \rho_0  }   E+ \rho_0u_x 
 F+   \rho_0u_y  G+ \frac{     \rho_0\rho_x}{\rho+\rho_0} H \\
 -  \frac{       h_0\rho_y}{\rho    + \rho_0  }   E+  \rho_0 v_x 
 F+ \rho_0 v_y G+   \frac{     \rho_0\rho_y}{\rho+\rho_0}H \\
   \frac{      \rho_0   }{ (\gamma-1) h_0  }  h_x F+
 \frac{      \rho_0   }{ (\gamma-1) h_0  }   h_y  G+ \frac{    
  \rho_0 ( u_x + v_y)   }{  h_0   }   H
\end{array} \right)  
=\begin{pmatrix} C_1\\C_2\\C_3\\C_4\end{pmatrix} 
.
\end{align*}
We apply the operator $\Lambda^\s$, where  $\s>1$ and
$s-2 < \s < s-1$, 
to the left hand side of \eqref{error-est} and then take the inner product 
with $\Lambda^\s \mathcal{E}$ to obtain  
\begin{align}
\langle \Lambda^{\sigma} \left( A_0(U^{\omega,n})\mathcal{E}_t \right), 
\Lambda^{\sigma} \mathcal{E} \rangle 
= 
& - \langle \Lambda^{\sigma} \left( C(U )\mathcal{E} \right), \Lambda^{\sigma} 
\mathcal{E} \rangle \label{eq:gr1}\\
& - \langle \Lambda^{\sigma} \left( \diag(A_1(U^{\omega,n}))\mathcal{E}_x + 
\diag(B_1(U^{\omega,n}))\mathcal{E}_y \right), \Lambda^{\sigma} \mathcal{E} \rangle 
\label{eq:gr2}\\
& -\langle \Lambda^{\sigma} \left( A_R(U^{\omega,n})\mathcal{E}_x + 
B_R(U^{\omega,n})\mathcal{E}_y \right), \Lambda^{\sigma} \mathcal{E} \rangle 
\label{eq:gr3}\\
& -  \langle \Lambda^{\sigma} A_0(U^{\omega,n}) R , \Lambda^{\sigma} \mathcal{E}  
\rangle \label{eq:gr4},
\end{align}
where $\diag(A)$ denotes the diagonal part of a matrix $A$ and $A_R=A - \diag(A)$. 
We now bound the terms on the right hand side.

\smallskip
\noindent
{\bf Estimate for \eqref{eq:gr1}.}
We have 
\begin{align}
\langle \Lambda^\s C(U)\mathcal E,\Lambda^\s\mathcal E\rangle
=\langle \Lambda^\s C_1,\Lambda^\s E\rangle+
\langle \Lambda^\s C_2,\Lambda^\s F\rangle+
\langle \Lambda^\s C_3,\Lambda^\s G\rangle+
\langle \Lambda^\s C_4,\Lambda^\s H\rangle\,.
\end{align}
These terms are all estimated in a similar way. 
The Cauchy-Schwarz inequality   yields 
$$|\langle \Lambda^\s C_1,\Lambda^\s E\rangle|\leq \|\Lambda^\s C_1\|_{L_2}
\|\Lambda^\s E\|_{L_2}
=  \|C_1\|_{\s}\|E\|_{\s}\leq \|C_1\|_{\s}\|\mathcal E\|_{\s}
$$
and so on for the other three terms.
To estimate $\|C_i\|_{\s}$, we note that all of the $C_i$ are
of the form
\begin{equation} \label{Cdef} 
C_i=a_1 E+a_2F+a_3G+a_4H\,,\end{equation}
where, up to constant multiples, each $a_j$ consists of a derivative, or sum
of derivatives, of components of $U$, in some cases divided by $\rho+\rho_0$.
So, taking $C_2$ as an example, and looking at the first summand, we have
\begin{equation} \label{a1}
\|a_1E\|_\s=h_0\left\|\frac{\rho_x}{\rho+\rho_0}E\right\|_\s\leq h_0
\left\|\frac{\rho_x}{\rho+\rho_0}\right\|_\s\|E\|_\s
\leq h_0 \left\|\frac{\rho_x}{\rho+\rho_0}\right\|_\s\|\mathcal E\|_\s\,,
\end{equation}
where we have used the algebra property, Lemma \ref{srp}.
Now we use Lemma \ref{Reciprocal estimate} to obtain
$$ \left\|\frac{\rho_x}{\rho+\rho_0}\right\|_\s\leq C(\s,\rho_0)
\big(1+\|\rho\|_\s^\s\big)\|\rho_x\|_\s \leq C \|\rho\|_{\s+1}\,,
$$
since $\|\rho_x\|_\s\leq \|\rho\|_{\s+1}$, and from Corollary \ref{moresolsize} 
we can absorb all the other factors into a
constant that depends on $\s$, $\rho_0$ and on the
$H^\s$ bound on $\rho$.
Finally, estimating $\|\rho\|_{\s+1}\leq Cn^{\s -s+1}$ as in Corollary \ref{moresolsize}, 
and treating the other terms in \eqref{Cdef} in the same way as \eqref{a1}, we have
$$
\langle \Lambda^\s C(U)\mathcal E,\Lambda^\s\mathcal E\rangle
  \le   C  n^{\s - s +1 } \|     \mathcal{E}  \|_{ \s } ^2 \,,
$$
with a constant $C$ that depends upon $\rho_0$, $h_0$, $\gamma$ and $\s$. 
(Since $\|U\|_{\s+1}$ decreases with $n$, we can eliminate the dependence of
the constant on $U$.)

\smallskip

\noindent{\bf Estimate of \eqref{eq:gr2}.} We have (up to a sign)
\begin{align*}
 \eqref{eq:gr2} 
&   = 
\langle \Lambda^{\sigma} \left( \diag(A_1(U^{\omega,n}))\mathcal{E}_x + \diag(B_1(U^{\omega,n}))\mathcal{E}_y \right), \Lambda^{\sigma} \mathcal{E} \rangle 
\\
& = 
\left \langle  
\Lambda^{\sigma}\left (  
\frac{ h_0}{ \rho_0}  u^{\omega,n}  E_x   + 
\frac{ h_0}{ \rho_0}  v^{\omega,n}  E_y  
 \right) 
,
\Lambda^{\sigma} E
\right\rangle _{L^2} 
 + 
\left \langle  
\Lambda^{\sigma} \left(  
 \rho_0  u^{\omega,n} F_x +  \rho_0  v^{\omega,n} F_y 
\right) 
,
\Lambda^{\sigma} F
\right\rangle _{L^2} 
\\ & + 
 \left\langle  
\Lambda^{\sigma} \left( 
 \rho_0  u^{\omega,n} G_x +  \rho_0  v^{\omega,n} G_y 
\right) 
,
\Lambda^{\sigma} G
\right\rangle _{L^2} 
 + 
 \left\langle  
\Lambda^{\sigma} \left( 
\frac{  \rho_0  u^{\omega,n} } {(\gamma-1) h_0 }  H_x 
+\frac{  \rho_0  v^{\omega,n} } {(\gamma-1) h_0 }  H_y 
\right) 
,
\Lambda^{\sigma} H
\right\rangle _{L^2}  .
\end{align*}
The eight terms in this expression are similar to each other; we show
how the first is estimated. 
Ignoring the constant $h_0/\rho_0$, consider 
$$
I_1  \equiv
\left \langle  
\Lambda^{\sigma}\left (  
u^{\omega,n}  E_x  
 \right) 
,
\Lambda^{\sigma} E
\right\rangle _{L^2}  = \int_{\rr^2}  \Lambda^\s \left (  
u^{\omega,n}  E_x    \right)\Lambda^{\sigma} E\, dxdy . 
$$
This can be written as (recall equation \eqref{commdef} for the definition of the commutator)
$$
I_1   = \int_{\rr^2}  \left( \left[ \Lambda^\s, \left (  
 u^{\omega,n} \right) \right]  E_x     +\left (  
 u^{\omega,n} \right)  \Lambda^\s \p_x E   \right) \Lambda^{\sigma} E\, dxdy . 
$$
We split this integral into two pieces, and apply the Cauchy-Schwarz estimate 
to the first term, to obtain 
$$
I_1   \le  
 \left\|  \left[ \Lambda^\s, \left (  
 u^{\omega,n} \right) \right]  E_x \right\|_{L^2} \| E\|_{\s}   
+
\left|\int_{\rr^2}   
 u^{\omega,n}  \Lambda^\s \p_x E    \Lambda^{\sigma} E\, dxdy \right|\,. 
$$
Now, the
Kato-Ponce commutator estimate, \eqref{KPest}, applied to the
first factor gives
\begin{align*}
\left\|  \left[ \Lambda^\s, \left (  
 u^{\omega,n} \right) \right]  E_x \right\|_{L^2} \leq &
 C(\s)\left( \|u_x^{\omega,n}\|_{L^\infty}\|\Lambda^{\s-1}E_x\|_{L_2}
 +\|\Lambda^\s u^{\omega,n}\|_{L_2}\|E_x\|_{L^\infty}\right)\\
\leq &C(\s)\left(\|u^{\omega,n}\|_{\s+1}\|\Lambda^\s E\|_{L_2}
 +\|u^{\omega,n}\|_\s\|E_x\|_{L^\infty}\right)\,,
 \end{align*} 
 using the Sobolev embedding theorem, Theorem \ref{sob}, which applies
 here since  $\s+1>2$.
 Since we can replace $\|u^{\omega,n}\|_\s$ by $\|u^{\omega,n}\|_{\s+1}$,
 and, using the same Sobolev embedding, replace $\|E_x\|_{L^\infty}$
 by $\|E\|_\s$, we obtain
 $$ \left\|  \left[ \Lambda^\s, \left (  
 u^{\omega,n} \right) \right]  E_x \right\|_{L^2} \leq C(s) \|u^{\omega,n}\|_{\s+1} \|E\|_\s^2
\,. $$
For the second term, integration by parts followed by H\"older's inequality yields 
\begin{align*}
\left|\int_{\rr^2}    
 u^{\omega,n}   \Lambda^\s \p_x E    \Lambda^{\sigma} E\, dxdy \right|
=&\left|\frac{1}{2}\int_{\rr^2}  
 u^{\omega,n}_x \big(\Lambda^\s E\big)^2  \, dxdy \right|
 \leq \frac{\|u^{\omega,n}_x\|_{L^\infty} }{2}\int_{\rr^2}  
 \big(\Lambda^\s E\big)^2  \, dxdy\\
= & \frac{1}{2}\|u^{\omega,n}_x\|_{L^\infty}\|E\|_\s^2   \,,
\end{align*}
and we get a bound similar to the first term, so that
$$
I_1   \le C
\left\|  
u^{\omega,n} \right\|_{\s+1}  \| E\|_{\s}   ^2 
 \le  Cn ^{\s+1-s} \| 
E\|_{\s}  ^2\,,
$$ 
from Corollary \ref{moresolsize}, with $C= C(\s)$.
Proceeding the same way with the other seven terms, we obtain
$$
\left|\langle \Lambda^{\sigma} \left( \diag(A_1(U^{\omega,n}))\mathcal{E}_x 
+ \diag(B_1(U^{\omega,n}))\mathcal{E}_y \right), \Lambda^{\sigma} \mathcal{E} \rangle 
\right| \leq Cn ^{\s+1-s} \| \mathcal E\|_{\s}^2,
$$
with the constant depending on on $\rho_0$, $h_0$, $\gamma$ and $\s$.

\smallskip
\noindent {\bf Estimate of \eqref{eq:gr3}.} 
Inserting the off-diagonal elements  of $A_1$ and $B_1$ from Section \ref{symmsyst}
(note that they are all constant
since $h^{\omega,n}=0=\rho^{\omega,n}$), we have 
\begin{align*}
- \eqref{eq:gr3}  
& = 
\left \langle  
\Lambda^{\sigma}\left (  
 h_0 F_x   +  h_0 G_y
 \right) 
,
\Lambda^{\sigma} E
\right\rangle _{L^2} 
  + 
\left \langle  
\Lambda^{\sigma} \left( 
 h_0  E_x +  \rho_0 H_x
\right) 
,
\Lambda^{\sigma} F
\right\rangle _{L^2} 
\\ & + 
 \left\langle  
\Lambda^{\sigma} \left( 
 h_0  E_y + \rho_0 H_y 
\right) 
,
\Lambda^{\sigma} G
\right\rangle _{L^2} 
 + 
 \left\langle  
\Lambda^{\sigma} \left(  \rho_0 F_x    + \rho_0 G_y 
\right) 
,
\Lambda^{\sigma} H
\right\rangle _{L^2}  .
\end{align*}
Writing the above as an integral and rearranging terms gives
\begin{align*}
- \eqref{eq:gr3}  
 = &
\int_{\rr^2} h_0\bigg(
\p_x\big(\Lambda^{\sigma}E\big) 
 \Lambda^{\sigma} F  + \Lambda^{\sigma} E\,\p_x
\big(\Lambda^{\sigma} F\big)+
\p_y\big(\Lambda^{\sigma}E\big) 
 \Lambda^{\sigma} G  + \Lambda^{\sigma} E\,\p_y
\big(\Lambda^{\sigma} G\big)\bigg)
\,dx\,dy 
\\ & + 
\int_{\rr^2} \rho_0\bigg(
\p_x\big(\Lambda^{\sigma}F\big) 
 \Lambda^{\sigma} H + \Lambda^{\sigma} F\,\p_x
\big(\Lambda^{\sigma} H\big)+
\p_y\big(\Lambda^{\sigma}G\big) 
 \Lambda^{\sigma} H  + \Lambda^{\sigma} G\,\p_y
\big(\Lambda^{\sigma} H\big)\bigg)
\,dx\,dy \\
=&
h_0\int_{\rr^2} 
\p_x\big(\Lambda^{\sigma}E \,
 \Lambda^{\sigma} F\big) +
\p_y\big(\Lambda^{\sigma}E\, 
 \Lambda^{\sigma} G \big) 
\,dx\,dy \\
& + 
\rho_0\int_{\rr^2} 
\p_x\big(\Lambda^{\sigma}F\,
 \Lambda^{\sigma} H\big) +
\p_y\big(\Lambda^{\sigma}G\, 
 \Lambda^{\sigma} H \big) 
\,dx\,dy 
\end{align*}
and therefore they all integrate to zero.  

\smallskip
\noindent{\bf Estimate of \eqref{eq:gr4}.} 
Since $A_0$ is diagonal and $A_0(U^{\omega,n})$ is constant, we have 
\begin{align*}
- \eqref{eq:gr4}   & =   \langle \Lambda^{\sigma} A_0(U^{\omega,n}) R , \Lambda^{\sigma} \mathcal{E}  \rangle  
  = 
\langle \rho_0\Lambda^\s R_2,\Lambda^\s F\rangle +
\langle \rho_0\Lambda^\s R_3,\Lambda^\s G\rangle \end{align*}
The Cauchy-Schwarz inequality   yields 
\begin{align*}
|\eqref{eq:gr4}| & \le \rho_0 \| R\|_{\s} \|\mathcal E\|_{\s} \,.
\end{align*}
 Combining the estimates for \eqref{eq:gr1}, \eqref{eq:gr2}, \eqref{eq:gr3} and \eqref{eq:gr4} we have 
\begin{align}
\langle \Lambda^{\sigma} \left( A_0(U^{\omega,n})\mathcal{E}_t \right), \Lambda^{\sigma} \mathcal{E} \rangle 
\le C
 n^{\s+1-s} \|  \mathcal{E} \|_{H^\s }^2 
 +C \| R\|_{\s} \|\mathcal E\|_{\s}  , \label{almost-gronwall}
\end{align} 
where the   constants depend upon $\rho_0, h_0, \gamma$ and $\s$.

We show that the residue $R$ satisfies the following estimate.
\begin{proposition}\label{residue-estimate}  If $ \max\{1,s-2\}< \s < s-1$, then 
$
\|R\|_\s \lesssim n^{\delta-2   }
$.
\end{proposition} 
\begin{proof}  \eqref{Rdef} the nonzero components of $R$ are
$$ 
\begin{pmatrix} 
R_2 \\
R_3  
\end{pmatrix}  
= 
\begin{pmatrix}
 \p_t u^{\omega, n} +  u^{\omega, n} \p_x u^{\omega, n}   + v^{\omega, n}  \p_y u^{\omega, n} \\
 \p_t v^{\omega, n} +  u \p_x v^{\omega, n} + v^{\omega, n}  \p_y v^{\omega, n}     
\end{pmatrix} \,  .
$$

\subsection{Estimating $R_2$} 
We have (omitting the superscripts for brevity)
\begin{align*} R_2=& u_t+uu_x+vu_y=
(u_1+u_2)_t+(u_1+u_2)(u_1+u_2)_x+(v_1+v_2)(u_1+u_2)_y\\
=&\overbrace{u_{2,t}}^{\fbox{0}}+\overbrace{u_1u_{1,x}}^{\fbox{1}}+
\overbrace{u_1u_{2,x}}^{\fbox{2}}
+ \overbrace{u_2u_{1,x}}^{\fbox{3}} +\overbrace{u_2u_{2,x}}^{\fbox{4}}
+\overbrace{v_1u_{1,y}}^{\fbox{5}}+
\overbrace{v_1u_{2,y}}^{\fbox{6}}+\overbrace{v_2u_{1,y}}^{\fbox{7}}+
\overbrace{v_2u_{2,y}}^{\fbox{8}}\,.
\end{align*}
Now, three of these terms are zero by design, since 
$\supp u_2=\supp v_2=\supp S$ and
$\phi_2'=0=\phi_2''$ for $y\in\supp S$:
\begin{align*} \fbox{2} &\equiv \frac{\omega}{n}\phi_1\left(\frac{x}{n^\delta}\right)
\phi_2'\left(\frac{y}{n^\delta}\right)\frac{1}{n^{\delta+s+1}}\p_{xy} S=0\\
\fbox{3}&\equiv u_2u_{1,x}=\frac{1}{n^{\delta+s+1}}\p_y S
\frac{\omega}{n^{1+\delta}}\phi_1'\left(\frac{x}{n^\delta}\right)
\phi_2'\left(\frac{y}{n^\delta}\right)=0\\
\fbox{7}&\equiv v_2u_{1,y} =-\frac{1}{n^{\delta+s+1}}\p_x S\left(-\frac{\omega}{n^{1+\delta}}
\right)
\phi_1\left(\frac{x}{n^\delta}\right)\phi_2''\left(\frac{y}{n^\delta}\right)=0\,;
\end{align*}
and another term takes a simpler form, since $\phi_1' \equiv 1\equiv \phi_2$ on
the support of $S$:
$$ \fbox{6}\equiv v_1u_{2,y} = -\frac{\omega}{n}
\phi_1'\left(\frac{x}{n^\delta}\right)\phi_2\left(\frac{y}{n^\delta}\right)\frac{1}{n^{\delta
+s+1}}\p_y S =-\frac{\omega}{n^{\delta+s+2}}\p_y^2 S\,.
$$
From the form of the low-frequency and high-frequency terms, it is clear that
differentiation of $u_1$ or $v_1$ with respect to either $x$ or $y$ improves
the result by a factor of $n^{-\delta}$, as does differentiation of $S$ with respect
to $x$; however, differentiation of $S$ with respect to $y$ introduces a term
with an additional multiplicative factor of $n$. 
The amplitudes of the low- and high-frequency terms have been balanced
so that the largest contributions due to this, in $\fbox{0}$ and $\fbox{6}$,
cancel each other. 
This is exhibited in the proof of
%%%%%%%
\begin{lemma} [Crucial cancellation]\label{crucial-cancel}
If $\varphi _1'\equiv 1$ on $\text{\emph{supp}}\,\psi'$ and
$\varphi _2\equiv 1 $   on  $\text{\emph{supp}}\,\psi$, then
\begin{align}\label{line1} \notag
u_{2,t}+v_1u_{2,y} \equiv& \frac{1}{n^{\delta +s+1}}\p_y\left(\p_t-\frac{\omega}{n}\p_y
\right) S \\=&
-\frac{\omega}{n^{2+2\delta+s}}
\psi\left(\frac{x}{n^\delta}\right)\left[
\frac{1}{n^\delta}\psi''\left(\frac{y}{n^\delta}\right)\sin(ny+\omega t)
+n\psi'\left(\frac{y}{n^\delta}\right)\cos(ny+\omega t)\right]\,,
\end{align}
and hence
$$ \left\|\fbox{\rm 0}+\fbox{\rm 6}\right\|_\s \leq Cn^{\s -\delta -s-1}\,.$$
\end{lemma}
\begin{proof} 
Using $
S(x,y,t) = \psi ( n^{-\delta } x) \psi(n^{-\delta} y) \sin( n y + \omega t)
$, 
we calculate
\begin{align*}
\left(\p_t-\frac{\omega}{n}\p_y\right) S=&\psi\left(\frac{x}{n^\delta}\right)
\left(\p_t-\frac{\omega}{n}\p_y\right)\left(\psi\left(\frac{y}{n^\delta}\right)\sin(ny+\omega t)
\right)\\
=&\psi\left(\frac{x}{n^\delta}\right)\left[\sin(ny+\omega t)\left(\p_t-\frac{\omega}{n}\p_y\right)
\psi\left(\frac{y}{n^\delta}\right)+\psi\left(\frac{y}{n^\delta}\right)
\left(\p_t-\frac{\omega}{n}\p_y\right)\sin(ny+\omega t)\right]\\
=&\psi\left(\frac{x}{n^\delta}\right)\left[\sin(ny+\omega t)\left(\p_t-\frac{\omega}{n}\p_y\right)
\psi\left(\frac{y}{n^\delta}\right)\right]\\
=&-\frac{\omega}{n^{1+\delta}}
\psi\left(\frac{x}{n^\delta}\right)\psi'\left(\frac{y}{n^\delta}\right)\sin(ny+\omega t)\,.
\end{align*}
From this we obtain \eqref{line1}.
Now it is a direct application of estimate \eqref{cos_estimate} to complete the proof. 
\end{proof}
To complete the estimate for the $H^\s$ norm of $R_2$,
we estimate the norms of $S$ and its derivatives.
From
$$ S=\psi\left(\frac{x}{n^\delta}\right)\psi\left(\frac{y}{n^\delta}\right)\sin(ny+\omega t)$$
and Lemma \ref{hs-estimates} we have $\|S\|_\s\lesssim n^{\s+\delta}$.
Since differentiation with respect to $x$ scales the expression by $n^{-\delta}$
and differentiation with respect to $y$ scales it by $n$ (where we ignore the
lower order contribution), we have
\begin{equation} \label{Sests}
  \| \p_x   S   \|_{\s} 
 \lesssim n^{\s}  \,,\quad
 \|  \p_y S   \|_{\s}  \lesssim  n^{ \s +\delta+1}  ,
\quad 
  \| \p_x  \p_y  S    \|_{\s}   
 \lesssim n^{ \s+1  }  ,
\quad 
   \|\p_y ^2   S   \|_{\s} \lesssim n^{\s+\delta +2} \,.
 \end{equation}
 We also note the $H^\s$ bounds on $u_1$ and $v_1$ and their derivatives:
 \[ \|u_1\|_\s =\left\|\frac{\omega}{n}\phi_1\left(\frac{x}{n^\delta}\right)
 \phi_2'\left(\frac{x}{n^\delta}\right)\right\|_\s
 \lesssim n^{\delta -1}\,,\quad  \|u_{1,x}\|_\s \lesssim \frac{1}{n}\,,
 \]
 and the same bounds hold for $v_1$ and for the $y$ derivatives.
With this we can find the remaining bounds for $R_2$:
\begin{align*}
\left\|\fbox{1}\right\|_\s&=\|u_1u_{1,x}\|_\s\lesssim n^{\delta -2}\,,\\
\left\|\fbox{4}\right\|_\s&=\|u_2u_{2,x}\|_\s=\frac{1}{(n^{\delta+s+1})^2}\|S_yS_{xy}\|_\s
\lesssim n^{-2s+2\s-\delta}\,,\\
\left\|\fbox{5}\right\|_\s&=\|v_1u_{1,y}\|_\s\lesssim n^{\delta -2}\,,\\
\left\|\fbox{8}\right\|_\s&=\|v_2u_{2,y}\|_\s=\frac{1}{(n^{\delta+s+1})^2}\|S_xS_{yy}\|_\s
\lesssim n^{-2s+2\s-\delta}\,.
\end{align*}
Combining this with Lemma \ref{crucial-cancel}, we find the $H_\s$ norm of $R_2$
to be bounded by $n^\alpha$ where
$$ \alpha = \max\{\delta-2, -2(s-\s)-\delta, \s -\delta -s-1\}\,.$$
Since $\s<s-1$, if we now choose $\delta \ll1$, the largest exponent is $\delta -2$, 
so we have
\begin{equation} \label{r2est}
\| R_2 \|_{\s}  \lesssim
 n^{  \delta -2} \, . 
\end{equation}

\subsection{Estimating $R_3$} 
This goes the same way.
(Again we omit the superscripts.)
\begin{align*} R_3=& v_t+uv_x+vv_y=
(v_1+v_2)_t+(u_1+u_2)(v_1+v_2)_x+(v_1+v_2)(v_1+v_2)_y\\
=&\overbrace{v_{2,t}}^{\fbox{0}}+\overbrace{u_1v_{1,x}}^{\fbox{1}}+
\overbrace{u_1v_{2,x}}^{\fbox{2}}
+ \overbrace{u_2v_{1,x}}^{\fbox{3}} +\overbrace{u_2v_{2,x}}^{\fbox{4}}
+\overbrace{v_1v_{1,y}}^{\fbox{5}}+
\overbrace{v_1v_{2,y}}^{\fbox{6}}+\overbrace{v_2v_{1,y}}^{\fbox{7}}+
\overbrace{v_2v_{2,y}}^{\fbox{8}}\,.
\end{align*}
Because $u_1$, $v_{1,x}$ and $v_{1,y}$ are zero on the support of $S$,
we find that the terms $\fbox{2}$, $\fbox{3}$ and $\fbox{7}$ are again zero
and (since $v_1$ is constant on $\supp S$), $\fbox{6}$ reduces to
$-\omega v_{2,y}/n$.
This again gives us a cancellation between the highest order terms in
$\fbox{0}$ and $\fbox{6}$ (we do not actually need it in the case of
$R_3$ since the largest terms are already smaller by a factor of $n$).
Specifically, using the identity in the proof of Lemma \ref{crucial-cancel},
\begin{align*}
 \fbox{0} + \fbox{6} &= v_{2,t}+v_1v_{2,y}=v_{2,t}-\frac{\omega}{n}v_{2,y}
=-\frac{1}{n^{\delta +s+1}}\p_x\left(\p_t-\frac{\omega}{n}\p_y\right)S\\
&=-\frac{1}{n^{\delta +s+1}}\p_x\left(
-\frac{\omega}{n^{1+\delta}}
\psi\left(\frac{x}{n^\delta}\right)\psi'\left(\frac{y}{n^\delta}\right)\sin(ny+\omega t)
\right)\\
&= \frac{1}{n^{3\delta +s+2}}
\psi'\left(\frac{x}{n^\delta}\right)\psi'\left(\frac{y}{n^\delta}\right)\sin(ny+\omega t)\,,
\end{align*}
and so
$$ \left\| \fbox{0} + \fbox{6}\right\|_\s\lesssim n^{-2\delta -s+\s-2}\,.$$
The estimates for the remaining terms are straightforward, as in the 
estimates for $R_2$.
We use \eqref{Sests} and we need also $\|S_{xx}\|_\s\lesssim n^{\s-\delta}$:
\begin{align*}
\left\|\fbox{1}\right\|_\s&=\|u_1v_{1,x}\|_\s\lesssim n^{\delta -2}\,,\\
\left\|\fbox{4}\right\|_\s&=\|u_2v_{2,x}\|_\s=\frac{1}{(n^{\delta+s+1})^2}\|S_yS_{xx}\|_\s
\lesssim n^{-2s+2\s-2\delta-1}\,,\\
\left\|\fbox{5}\right\|_\s&=\|v_1v_{1,y}\|_\s\lesssim n^{\delta -2}\,,\\
\left\|\fbox{8}\right\|_\s&=\|v_2v_{2,y}\|_\s=\frac{1}{(n^{\delta+s+1})^2}\|S_xS_{xy}\|_\s
\lesssim n^{-2s+2\s-2\delta-1}\,.
\end{align*}
Once again, the largest exponent is $\delta -2$, and so
\begin{align} \label{r3est}
\| R_3 \|_{\s}  & \lesssim 
 n^{ \delta -2  } \,. 
\end{align} 
 Combining estimates \eqref{r2est} and \eqref{r3est}
completes the proof of Proposition \ref{residue-estimate}.
\end{proof} 

To complete the proof of Theorem \ref{bigdeal},
first notice that from the definition $A_0 (U^{\omega,n} ) \ge c I$ 
for some positive constant $c$, and $A_0(U^{\omega,n})$ is a constant
matrix. 
Therefore, the $L^2$ inner product $\langle A_0 (U^{\omega,n} )  
V, V \rangle$ defines an equivalent norm. 
Thus,
\begin{align}
\frac{d}{dt} \| \mathcal{E}   \|_{\s}   ^2 & = \frac{d}{dt}
 \langle \Lambda^\s  \mathcal{E}   ,\Lambda^\s  \mathcal{E} \rangle
  \approx 
  \frac{d}{dt} \langle A_0 (U^{\omega,n} )  \Lambda^\s  \mathcal{E}   , 
   \Lambda^\s  \mathcal{E} \rangle
\end{align}
Applying the derivative, we have 
\begin{align} 
  \frac{d}{dt} \langle A_0 (U^{\omega,n} )   \Lambda^\s  \mathcal{E}     , 
   \Lambda^\s  \mathcal{E} \rangle
  &= 2
\langle A_0 (U^{\omega,n} )  \Lambda^\s  \mathcal{E}  _t   , 
 \Lambda^\s  \mathcal{E} \rangle
+ 
\langle A_0 (U^{\omega,n} ) _t  \Lambda^\s  \mathcal{E}   ,  \Lambda^\s  \mathcal{E} \rangle
\\ & = 
  2\Big\langle  \Lambda^\s  \Big( A_0 (U^{\omega,n} )   \mathcal{E}  _t  \Big) ,  
  \Lambda^\s  \mathcal{E}   \Big\rangle
 \notag
\end{align}
(since $ A_0 (U^{\omega,n} ) $ is constant). 
This quantity was estimated in Section \ref{section3.4}; substituting inequality 
\eqref{almost-gronwall} and applying Proposition \ref{residue-estimate} we have 
\begin{align} 
2\|\mathcal E\|_\s\frac{d}{dt}\|\mathcal E\|_\s\approx
  \frac{d}{dt} \langle A_0 (U^{\omega,n} )   
   \Lambda^\s  \mathcal{E}     ,  \Lambda^\s  \mathcal{E} \rangle
& \lesssim  
n^{\s+1-s} \|  \mathcal{E} \|_{H^\s }^2 
 + n^{\delta -2 } \|\mathcal E\|_{\s}  \, .   \label{eq:341} 
\end{align} 
Dividing by $\|\mathcal E\|_\s$ in \eqref{eq:341}
gives the first inequality stated in the theorem. 
We apply  
Gr\"onwall's inequality, \cite[p 24]{hartman}. 
The Gr\" onwall estimate for
$$ z'(t)\leq a z(t) + b\,,\quad z(0)=0\,,$$
is 
$$ z(t)\leq \frac{b}{a}\bigg( e^{at}-1\bigg)\,.$$
Since here  
$a\simeq n^{\s -s + 1}< C$, the upper bound for $e^{at}$ is a constant
that depends only on $T^*$, the time interval on which we are tracking the
solution, and with $b\simeq n^{\delta -2}$, then $b/a$ gives  the estimate in the theorem.
\end{proof} 
This completes the proof of the uniformity (in $n$) of the approximation of
$U^{\omega,n}$ to the actual solution, $U_{\omega,n}$, for $\omega = 1$
and $\omega = -1$.
For $\varepsilon$ in equation \eqref{eqapp} we have $3-(s-\s)-\delta>2$.
 
\subsection{Nonuniform convergence} \label{noncvge}
 We are now prepared to complete the proof of nonuniform convergence,
the final item, (\ref{nu}), in the program. 
We use a fact we proved in \cite{KeTi}:
 For a range of $\tau> s$, (specifically $s<\tau\leq \lfloor s\rfloor +1$, 
 where $\lfloor \cdot \rfloor$ is the greatest integer function),
the error in the $H^\tau$ norm is bounded by 
\begin{equation} \label{bignorm}
\| \mathcal E \|_{ {\tau} }  \lesssim n^{\tau - s} \,. 
\end{equation}
This uses the form of $\mathcal E$ and the bound in Lemma \ref{hs-estimates}.
Interpolation yields an estimate for the error in the $s$ norm. 
\begin{lemma}\label{error}
For any $s>2$ and $n\gg1$, there exists an $\ee>0$ such that  
\begin{equation}
\| \mathcal E \|_{ {s} }  \lesssim n^{-\ee}    .
\end{equation}
\end{lemma}
\begin{proof}
From Proposition \ref{interp} we have
\begin{equation}
\big\| \mathcal E\big \|_{ {s} }  \leq\big \| \mathcal E\big \|_ \s^{\alpha}  
\big\| \mathcal E\big \|_\tau^{\beta}, \quad \textrm{where}
\quad \alpha= \frac{ \tau -s}{\tau - \s}\,,\quad \beta =\frac{  s -\s }{\tau - \s} \,.
\end{equation}
Using Theorem \ref{bigdeal} and \eqref{bignorm} we find 
\begin{equation}
\| \mathcal E \|_{ {s} }  \lesssim \big( n^{ \delta-3+ s - \s   }\big) ^{\alpha }  
\big( n^{\tau-s }\big)^{\beta }  = n^{( \tau - s)   (\delta - 3 + 2s-2\s) / (\tau - \s)}   .
\end{equation}
By  choosing $\max\{1, s-3/2+\delta \} < \s < s-1$, we obtain $\delta - 3 + 2s-2\s < 0$,
which completes the proof.  
\end{proof}

We are now ready to give the proof of Theorem \ref{intro-thm}. 
\begin{proof}%
We estimate the difference between two actual solutions by the triangle inequality 
\begin{align}
&\|U_{1,n}  - U_{-1,n} \|_{ s} \ge 
\|U^{1,n}  - U^{-1,n} \|_{ s} 
-
\|U^{1,n}  - U_{ 1,n} \|_{ s} 
-
\|U^{-1,n}  - U_{-1,n} \|_{ s} 
.
\end{align}
From Lemma \ref{error}, the last two terms tend to zero as 
$n\rightarrow \infty$, and therefore, tracking the terms that
do not tend to zero as $n\to\infty$,
\begin{align*}
 \|U_{1,n}  - U_{-1,n} \|_{ s} 
& \notag
 \ge 
\liminf_{n\rightarrow \infty} \|U^{1,n}  - U^{-1,n} \|_{ s}  
\\
& \notag \ge\lim_{n\rightarrow \infty}  \| 
n^{-\delta -s}\psi(n^{-\delta} x)\psi(n^{-\delta} y)  \left( \cos( n y +  t) - 
 \cos( n y - t) \right) \|_{s} 
\\
& \notag = \lim_{n\rightarrow \infty}  \|
 n^{-\delta -s}\psi(n^{-\delta} x)\psi(n^{-\delta} y)    \cos( n y    )\|_{s} | \sin(t)| 
\approx  |\sin(t)|.
\qedhere
\end{align*}
\end{proof}
%
%
%%%%%%%%%%%%%%%%

 \end{document}